\documentclass{amsart}

\usepackage{amssymb}
\usepackage{eqlist, color}
\usepackage{mathrsfs, textcomp}
\usepackage{graphicx}
\usepackage{bm}
\usepackage{enumitem}

\usepackage{amsmath,amsfonts,amsthm,latexsym}
\usepackage[all]{xy}

\newtheorem{theorem}{Theorem}[section]
\newtheorem{lemma}[theorem]{Lemma}
\newtheorem{proposition}[theorem]{Proposition}
\newtheorem{corollary}[theorem]{Corollary}
\newtheorem{facts}[theorem]{Facts}
\newtheorem{question}[theorem]{Question}
\theoremstyle{definition}
\newtheorem{definition}[theorem]{Definition}
\newtheorem{definitions}[theorem]{Definitions}
\newtheorem{example}[theorem]{Example}

\newtheorem{remark}[theorem]{Remark}

\newtheorem{defnfacts}[theorem]{Definitions and Facts}

\numberwithin{equation}{section}

\newcommand{\N}{\mathbb{N}}

\newcommand{\R}{\mathbb{R}}

\newcommand{\Wstar}{\mathbf{W}^*}

\newcommand{\pperp}{\perp\perp}
\newcommand{\Int}{{\rm int}}

\DeclareMathOperator{\cl}{cl}
\DeclareMathOperator{\coz}{coz}

\DeclareMathOperator{\M}{(M)}
\DeclareMathOperator{\Y}{(Y)}
\DeclareMathOperator{\wk}{\Wstar k}
\DeclareMathOperator{\CR}{(CR)}
\DeclareMathOperator{\HA}{(HA)}
\DeclareMathOperator{\Proj}{(Pr)}
\DeclareMathOperator{\aN}{\alpha\mathbb{N}}

\newcommand{\Mart}{Mart\'{i}nez }

\theoremstyle{definition}
\theoremstyle{definition}
\theoremstyle{definition}
\theoremstyle{definition}
\theoremstyle{definition}
\theoremstyle{definition}
\theoremstyle{remark}
\theoremstyle{definition}
\theoremstyle{definition}

\begin{document}
\title[${\bf{W}^*}$: cozero-sets and ideals]{Archimedean $\ell$-groups with strong unit: cozero-sets and coincidence of types of ideals}

\author{Papiya Bhattacharjee, \quad Anthony W. Hager, \quad Warren Wm. McGovern, \quad Brian Wynne   }
\address{Department of Mathematical Sciences, Charles E. Schmidt College of Science, Florida Atlantic University, Boca Raton, FL 33431, USA}
\address{Department of Mathematics, Oregon State University, Corvallis, OR 97331-4605, USA}
\address{Wilkes Honors College, Florida Atlantic University, Jupiter, FL 33458, USA}
\email{pbhattacharjee@fau.edu (P. Bhattacharjee)}
\email{}
\email{warren.mcgovern@fau.edu (W. Wm. McGovern)}
\subjclass[2010]{}

%%% ----------------------------------------------------------------------

%\let\thefootnote\relax\footnote{*Corresponding author}

\begin{abstract}
$\Wstar$ is the category of the title. For $G \in \Wstar$, we have the canonical compact space $YG$, and Yosida representation $G \leq C(YG)$, thus, for $g \in G$, the cozero-set $\coz(g)$ in $YG$. The ideals at issue in $G$ include the principal ideals and polars, $G(g)$ and $g^{\perp \perp}$, respectively, and the $\Wstar$-kernels of $\Wstar$-morphisms from $G$. The ``coincidences of types" include these properties of $G$: (M) Each $G(g) = g^{\perp \perp}$; (Y) Each $G(g)$ is a $\Wstar$-kernel; (CR) Each $g^{\perp \perp}$ is a $\Wstar$-kernel (iff each $\coz(g)$ is regular open). For each of these, we give numerous ``rephrasings", and examples, and note that (M) = (Y) $\cap$ (CR). This paper is a companion to a paper in preparation by the present authors, which includes the present thrust in contexts less restrictive and more algebraic. Here, the focus on $\Wstar$ brings topology to bear, and sharpens the view.
\end{abstract}
%%%------------------------------------------------------------------

\maketitle

%--------------------------------------------------------------
\thispagestyle{empty}

\section{Introduction}\label{sect0}

``$\ell$-group" means ``lattice-ordered group", and the theory is exposed in \cite{D95}, \cite{AF89}, \cite{BKW77}, where undefined terms may be found.

\begin{definitions}
In an $\ell$-group $G$, let us say Abelian, an \emph{ideal} $I$ is the kernel of homomorphism (i.e., an $\ell$-group homomorphism) out of $G$, which means $I$ is a convex sub-$\ell$-group of $G$. Here, for each $g \in G$, we have the principal ideal $G(g) = \{ f \in G \mid \exists n \in \mathbb{N}, \vert f \vert \leq n\vert g \vert \}$ ($\mathbb{N}$ is the positive integers), and the principal polar $g^{\pperp}$. $g^{\perp} = \{ f \in G \mid \vert f \vert \wedge \vert g \vert = 0 \}$, thus $g^{\pperp} = (g^{\perp})^{\perp}$ (where, for $S \subseteq G$, $S^{\perp} = \bigcap \{ s^{\perp} \mid s \in S \}$.

Now, the class $\M = \{ G \mid \forall g \in G, G(g) = g^{\pperp} \}$ is defined.

Next, an ideal $I$ in $G$ is maximal (in the usual sense) if and only if $G/I$ embeds into the real numbers $\R$ (H\"{o}lder's Theorem).

Then, the class $\Y = \{ G \mid G(g) \mbox{ is an intersection of maximal ideals} \}$. (The description of $\Y$ in the Abstract will develop).

The notations (M) and (Y) are explained in \cite{BHMWinf}.
\end{definitions}

Now we restrict to $\Wstar$: $(G,u) \in \Wstar$ means $G$ is Archimedean and $u \in G$ is a non-negative strong unit ($G(u) = G$). A $\Wstar$-morphism $(G,u) \xrightarrow{\varphi} (H,v)$ is an $\ell$-group homomorphism with $\varphi(u) = v$. $Y(G,u)$ is $\{ I \mid I \mbox{ is an ideal maximal for } u \notin I \}$. 

\begin{theorem}\label{1.2}
(Mostly Yosida; see \cite{BH89} or \cite{HR77}.)
\begin{itemize}
\item[(a)] Given the hull-kernel topology, $Y(G,u)$ is compact Hausdorff, and there is the $\Wstar$-embedding $(G,u) \hookrightarrow (C(Y(G,u)), 1)$ with the image of $(G,u)$ 0-1 separating closed sets of $Y(G,u)$. The space $Y(G,u)$ is unique for that data. We simply write ``$G \leq C(Y(G,u))$", and sometimes just $YG$ for $Y(G,u)$.
\item[(b)] If $(G,u) \xrightarrow{\varphi} (H,v)$ is a $\Wstar$-morphism, there is unique continuous $YG \xleftarrow{\tau} YH$ for which $\varphi(g)$ in $C(YG)$ is $\varphi(g) = g \circ \tau$, for each $g$. $\varphi$ is onto if and only if $\tau$ is one-to-one (thus a homeomorphic embedding, and we may say ``$YG \supseteq YH$").
\end{itemize}
\end{theorem}
We usually shall adopt the view $G \leq C(YG)$ (recall $YG = Y(G,u)$), as above. Then, for $g \in G$, $\coz(g) = \{ x \in YG \mid g(x) \neq 0 \}$, and $Z(g) = YG \setminus \coz(g)$, and $\coz(G) = \{ \coz(g) \mid g \in G \}$ (which is a base for the topology of $YG$).

If $(G,u) \xrightarrow{\varphi} (H,v)$ is a $\Wstar$-morphism, $\ker(\varphi) = \{ g \mid \varphi(g) = 0 \}$ is called a $\Wstar$-kernel, and $\wk(G)$ consists of all such $\ker(\varphi)$. The following is a synopsis from \cite[Section 2]{BH89} (an easy consequence of \ref{1.2}).

If $S$ is a closed subset of $Y(g,u)$, then a $\Wstar$-map is produced by restriction: $g \mapsto g \vert_S$, and this has kernel $I(S) = \{ g \mid Z(g) \supseteq S \}$. Observe that $S_1 \subseteq S_2$ if and only if $I(S_1) \supseteq I(S_2)$.

\begin{proposition}
$\wk(G)$ and $\{ S \mid S \mbox{ closed in } YG \}$ are complete lattices, and lattice anti-isomorphic: if $I \in \wk(G)$, then $I = I(S)$ for $S = \bigcap \{ Z(g) \mid g \in I \}$.
\end{proposition}

Armed with the 	``Yosida apparatus" just described, we discuss various types of ideals.

\begin{facts}\label{1.4}
Suppose $(G,u) \in \Wstar$, viewed as $G \leq C(Y(G,u))$.
\begin{itemize}
\item[(a)] The proper maximal ideals are, for $p \in Y(G,u)$, the $M_p \equiv I(\{ p \}) = \{ p \in G \mid g(p) = 0 \}$.

\item[(b)] The intersections of maximal ideals are exactly the $I(S)$ for $S$ closed in $Y(G,u)$, i.e., exactly the $\Wstar$-kernels (since $G \twoheadrightarrow G/I(S)$ is (isomorphic to) the map of restricting $G \twoheadrightarrow G \vert_S$).

\item[(c)] If $G(g) = I(S)$ ($S$ closed), then $S = Z(g)$.

\item[(d)] $g^{\pperp} = I(\cl \Int Z(g))$ ($\cl \Int$ is the topological closure of interior).

\item[(e)] (A ``regular open" (RO) subset of a topological space is a set $S = \Int \cl S$.)

For $g \in G$, $\coz(g)$ is RO if and only if $g^{\pperp} = I(Z(g))$.

\item[(f)] (A ``weak unit" in (any) $G$ is an element $v \geq 0$ for which $v^{\perp} =\{ 0 \}$ (or $v^{\pperp} = G$.)

$v$ is a weak (resp., strong) unit in $(G,u) \in \Wstar$ if and only if $\coz(v)$ is dense in $Y(G,u)$ (resp., $\coz(v) = Y(G,u)$).

\item[(g)] For all $g \in G$, $G(g) \subseteq I(Z(g)) \subseteq I(\cl \Int Z(g)) = g^{\pperp}$.
\end{itemize} 
\end{facts} 

\begin{proof}
(indications thereof)

(a) Clear.

(b) From (a), and $I(S) = I(\overline{S})$.

(c) Suppose $G(g) = I(S)$. Then $g \in I(S)$, so $Z(g) \supseteq S$. If $x \in Z(g)$, then $x \in Z(f)$ for every $f \in G(g)=I(S)$, and it follows that $x \in S$. Hence $Z(g) \subseteq S$.

(d) It is noted in \cite{HM96} that $g^{\pperp} = \{ f \in G \mid \coz(f) \subseteq \overline{\coz(g)} \}$. Then, a little topology (see \cite{E89}) gives (d).

(e) From (c) and (d).

(f) First: In general, $v^{\perp} = \{ f \in G \mid \coz(f) \cap \coz(v) = \emptyset \}$ and the statment about weak units follows. Second: It's easy that $\coz(v) = Y(G,u)$ if and only if $G = G(v)$.

(g) Follows from the preceding.
\end{proof}

\begin{definition}
In $\Wstar$, the class $\CR = \{ (G,u) \mid \forall g \in G, \coz(g) \mbox{ is RO} \}$.
\end{definition}

\begin{corollary}\label{1.6}
In $\Wstar$,
\begin{itemize}
\item[(a)] $G \in \Y$ if and only if $\forall g \in G$, $G(g) = I(Z(g))$.
\item[(b)] $G \in \CR$ if and only if $\forall g \in G$, $I(Z(g)) = g^{\pperp}$.
\item[(c)] $\M = \Y \cap \CR$.
\end{itemize}
\end{corollary}

\begin{proof}
All immediate from the various items in \ref{1.4}.
\end{proof}

%%%%%%%%%%%%%%%%%%%%%

\section{The class $\Y$ in $\Wstar$}

We remind that (in general) $G \in \Y$ means that each $G(g)$ is an intersection of maximal ideals.

\begin{theorem}\label{2.1}
For $(G,u) \in \Wstar$. the following are equivalent:
\begin{itemize}
\item[(1)] $G \in \Y$
\item[(2)] $\forall g \in G$, $G(g) = I(Z(g))$.
\item[(3)] $\forall g \in G$, $G(g)$ is a $\Wstar$-kernel.
\item[(4)] $\forall g \in G$, $G/G(g)$ is Archimedean.
\item[(5)] The Yosida embedding $G \leq C(Y(G,u))$ has the property: $\coz(f) \subseteq \coz(g)\Rightarrow f \in G(g)$. 
\end{itemize}
\end{theorem}

\begin{proof}
$(1) \Rightarrow (2)$ is \ref{1.6}(a).

$(2) \Rightarrow (3)$ from \ref{1.4}(c).

$(3) \Rightarrow (4)$. Obvious.

$(4) \Rightarrow (3)$. If $G/G(g)$ is Archimedean, the quotient map $G \twoheadrightarrow G/G(g)$  is a $\Wstar$-map (since any quotient preserves strong unit), so $G(g)$ is a $\Wstar$-kernel.

$(3) \Rightarrow (1)$. Any $\Wstar$-kernel is a intersection of maximal ideals (\ref{1.4}(b)). (3) says $Z(f) \supseteq Z(g) \Rightarrow f \in G(g)$, which is what (2) says. 
\end{proof}

\begin{remark}
Assuming nothing about $G$, \cite{MZ06} says $G \in \Y$ if and only if there is an embedding $G \hookrightarrow \R^X$ (for some $X$) with the property \ref{2.1}(5) (which specifies the embedding).
\end{remark}

%%%%%%%%%%%%%%%%%%%%%%%%%%%%%%%%%%%%%%%%%%%%%

\section{The class $\CR$ in $\Wstar$}

We remind the reader that $G \in \CR$ was defined in the Abstract as: Each $g^{\pperp}$ is a $\Wstar$-kernel, and this was shown in \ref{1.4} to be equivalent to:  Each $\coz(g)$ is RO -- regular open (hence ``CR", for ``cozero-regular"). 

We shall create a compendium of conditions equivalent  to ``$G \in \CR$", at least eight. To avoid the mind-numbing effect of a single list of these, we first make two sub-groups of these, of rather similar conditions, as the following two lemmas.

\begin{lemma}\label{3.1}
(Regularity conditions) For $G \in \Wstar$, the following are equivalent.
\begin{itemize}
\item[(R1)] $G \in \CR$.
\item[(R2)] $\forall f,g \in G$, if $\coz(f) \subseteq \overline{\coz(g)}$, then $\coz(f) \subseteq \coz(g)$.
\item[(R3)] $\forall g \in G$, $I(Z(g)) = I(\cl \Int Z(g))$ (i.e., $Z(g) = \cl \Int Z(g)$).
\item[(R4)] $\forall g \in G$, $I(Z(g)) = g^{\pperp}$. 
\end{itemize}
\end{lemma}

\begin{proof}
(R2) says that $\coz(g)$ is RO, i.e., (R1), and (R3) says $Z(g)$ is regular closed, which is equivalent to $\coz(g) = YG \setminus Z(g)$ is RO. From \ref{1.4}, $g^{\pperp} = I(\cl \Int Z(g))$, so (R4) if and only if (R3). 
\end{proof}

Now, in any $G$, an ideal $I$ is called a $d$-ideal if $g \in I$ implies $g^{\pperp} \subseteq I$. (For further comment on this, see \ref{3.4} below.) Observe that (i) any $g^{\pperp}$ is a $d$-ideal, and (ii) any intersection of $d$-ideals is a $d$-ideal.

\begin{lemma}\label{3.2}
($d$-ideal conditions) For $G \in \Wstar$, the following are equivalent:
\begin{itemize}
\item[(d1)] Each $\Wstar$-kernel of $G$ is a $d$-ideal.
\item[(d2)] $\forall g \in G$, $I(Z(g))$ is a $d$-ideal. 
\item[(d3)] $\forall p \in YG$, $M_p \equiv \{ g \in G \mid g(p) = 0 \}$ is a $d$-ideal. \\
(In our general notation, $M_p = I(\{ p \})$.)
\end{itemize}
\end{lemma}

\begin{proof}
The $\Wstar$-kernels are the $I(S)$, $S$ closed in $YG$.

(d3) $\Rightarrow$ (d1) by (ii) above, and (d1) $\Rightarrow$ (d2) obviously.

Suppose (d2) and $g \in M_p$. Then $p \in Z(g)$, so $M_p \supseteq I(Z(g))$. Thus $g^{\pperp} \subseteq I(Z(g)) \subseteq M_p$ and (d3) holds.
\end{proof}

(d1) of \ref{3.2} is a very curious property.

\begin{theorem}\label{3.3}
For $G \in \Wstar$, the following are equivalent.
\begin{itemize}
\item[(1)] $G \in \CR$ (any/all of the conditions in \ref{3.1} hold).
\item[(2)] For each $g \in G$, $I(Z(g))$ is a $d$-ideal (any/all of the conditions in \ref{3.2} hold).
\item[(3)] In $G$, each weak unit is strong (see \ref{1.4}(f)). 
\end{itemize}
\end{theorem}

\begin{proof}
$(1) \Leftrightarrow (2)$, in the form \ref{3.1}(R4) $\Leftrightarrow$ \ref{3.2}(d2): ($\Rightarrow$) $g^{\pperp}$ is always a $d$-ideal; ($\Leftarrow$) If $I(Z(g))$ is a $d$-ideal, then $g^{\pperp} \subseteq I(Z(g))$, thus $g^{\pperp} = I(Z(g))$ from \ref{1.4}.

$(1) \Rightarrow (3)$, in the form \ref{3.1}(R4) $\Rightarrow$ (3). Suppose $v$ is a weak unit, so $v^{\pperp} = G$. By (R4), $v^{\pperp} = I(Z(v))$, so $Z(v) = \emptyset$. Thus, $v$ is a strong unit (\ref{1.4}(f)).

$(3) \Rightarrow (1)$, in the form (3) $\Rightarrow$ \ref{3.1}(R2). Suppose (R2) fails, with $f,g > 0$: $\coz(f) \subseteq \overline{\coz(g)}$ but $\coz(f) \nsubseteq \coz(g)$. Choose $p \in \coz(f) \setminus \coz(g)$. Then, there is $h > 0$ with $h(p) = 0$ and $Z(f) \subseteq \coz(h)$ (since $G$ separates closed sets in $YG$). Let $k = h + g$. We have $k(p) = h(p) + g(p) = 0+0 = 0$, so $\coz(k) = \coz(h) \cup \coz(g) \neq YG$, and $k$ is not a strong unit. But $k$ is a weak unit, because $\coz(k)$ is dense in $YG$: if $\emptyset \neq U$ is open, and if $U \cap \coz(h) = \emptyset$, then $U \subseteq Z(h) \subseteq \coz(f)$, so $U \subseteq \overline{\coz(g)}$ and therefore $U \cap \coz(g) \neq \emptyset$.
\end{proof}

\begin{remark}\label{3.4}
For $(G,u) \in \bf{W}$ -- meaning $u$ is merely a weak unit -- we again have the Yosida space $Y(G,u)=YG$, and the Yosida representation now of 
\[
G \hookrightarrow D(YG) = \{ f \in C(YG,[-\infty,+\infty]) \mid f^{-1}(-\infty, +\infty) \mbox{ dense} \},
\]
so again the family $\coz(G)$, and $(G,u) \in \CR$ can be defined. (One sees that $(G,u) \in \CR$ if and only if $(G(u), u) \in \CR$ in $\Wstar$.)  This ``$\CR$ in $\bf{W}$"  is addressed in \cite{BM18}, focusing on $d$-ideals. The present treatment of $\CR$ owes much to that result. 
\end{remark}

%%%%%%%%%%%%%%%%%%%%%%%%%%%%%%%%%%%%%%%%%%%

\section{The class $\M$ in $\Wstar$}

We exhibit a few ways of saying $G \in \M$ within $\Wstar$.

(Of which there are many: \ref{1.6} shows $\M = \Y \cap \CR$, and Section 2 gives 5 ways of saying $G \in \Y$, Section 3 gives eight ways of saying $G \in \CR$, and  $5 \times 8 = 40$.)

\begin{corollary}\label{4.1}
For $G \in \Wstar$, the following are equivalent. 
\begin{itemize}
\item[(1)] $G \in \M$.
\item[(2)] $\forall g \in G$, $G(g)$ is a $d$-ideal.
\item[(3)] $\forall f,g \in G$, if $\coz(f) \subseteq \overline{\coz(g)}$, then $f \in G(g)$.
\item[(4)] $\forall g \in G$, $G/G(g)$ is Archimedean, and in $G$ every weak unit is strong.
\end{itemize}
\end{corollary}

\begin{proof}
(2) says $g^{\pperp} \subseteq G(g)$, thus $g^{\pperp} = G(g)$ by \ref{1.4}.

(3) is \ref{2.1}(5) plus \ref{3.1}(R2).

(4) is \ref{2.1}(4) plus \ref{3.3}(3).
\end{proof}

Now we relate $\M$ to two well-studied  properties (see \cite{D95}, inter alia).

\begin{defnfacts}\label{4.2}
(a) $G$ is hyperarchimedean $\HA$ if every quotient of $G$ is Archimedean.

$G$ is projectable $\Proj$ if each $g^{\pperp}$ is a $\ell$-group direct summand.

(b) For $G \in \Wstar$, as usual taking the ``Yosida view" $G \leq C(YG)$:
\begin{itemize}
\item[(i)] $G \in \HA$ if and only if each $\coz(g)$ is closed (see \cite{HK04}). \\
(One sees $\HA \subseteq \M$, e.g., by \ref{4.1}.)
\item[(ii)] $G \in \Proj$ if and only if each $\overline{\coz(g)}$ is open (see \cite{HM96}).
\end{itemize}
\end{defnfacts}

Bigard \cite{Bigard68} notes that in general $\ell$-groups $\Proj  \cap \M = \HA$.  This sharpens in $\Wstar$.

\begin{corollary}
In $\Wstar$, $\Proj \cap \CR = \HA$.
\end{corollary}

\begin{proof}
Evident using \ref{4.2}(b).
\end{proof}

%%%%%%%%%%%%%%%%%%%%%%%%%%%%%%%%%%%%%

\section{Some examples}

Within $\Wstar$, we shall: (1) indicate where $C(K)$ is in one class or another; (2) give examples showing that $\M$ is neither $\Y$ nor $\CR$ (the $\M = \Y \cap \CR$ is more or less sharp); (3) give examples in $\M$ (especially not $\HA$).

\begin{theorem}\label{5.1}
Suppose $K$ is compact.
\begin{itemize}
\item[(a)] The following are equivalent:
\begin{itemize}
\item[(i)] $C(K) \in \HA$;
\item[(ii)] $C(K) \in \Y$;
\item[(iii)] $C(K) \in \M$;
\item[(iv)] $K$ is finite.
\end{itemize}
\item[(b)] $C(K) \in \CR$ if and only if $K$ is an almost-$P$-space ($\equiv$ every nonempty zero-set has nonempty interior, or, is regular-closed (see \cite{L77}).
\end{itemize}
\end{theorem}

\begin{proof}
(a) If $K$ is finite, then all subsets are closed, thus all cozero-sets, so $C(K) \in \HA \subseteq \M \subseteq \Y$. If $K$ is not finite, a non-closed $\coz(g)$ is easily found (so $C(K) \notin \HA$), and then (take $g > 0$), $G(g^2)\neq I(Z(g^2))$ (with $\sqrt{g^2} = g$), so \ref{2.1}(2) fails and $C(K) \notin \Y$.

(b) Clear.
\end{proof}

Several examples to follow use the one-point compactification of $\N$, denoted $\aN = \N \cup \{ \alpha \}$.

\begin{example}\label{5.2}
$\Y \nRightarrow \CR$.

We give two (types of) examples.

(1) Let $G \leq C(\aN)$ be generated by $F \equiv \{ f \in C(\aN) \mid f(\aN) \mbox{ finite}\}$ and 
\[
v(x) = \left \{
\begin{array}{ll}
\frac{1}{x} & x \in \N \\
0 & x = \alpha.
\end{array}
\right.
\]
Note  that $g \in G$ can be put $(\ast)$ $g \dot{=} r + sv$ for some $r,s \in \R$, with $\dot{=}$ denoting ``eventually equal on $\N$". Here, $\coz(v) = \N$, so $v$ is a weak unit, not strong. Thus $G \notin \CR$. That $G \in \Y$ can be checked using any of \ref{2.1}, with some calculating using $(\ast)$.

(2) (A simplification of \cite[Example 5.5]{MZ06}). Let $G \leq C([0, +\infty])$ consist of the $g$  which are ``finitely piecewise linear" ($[0,+\infty] = [x_0,x_1] \cup [x_1, x_2] \cup \cdots \cup [x_{n-1},x_n]$, with $x_0 = 0$, $x_n = +\infty$, and $g$ linear on each $[x_i,x_{i+1}]$). Take $g \in G$ such that
\[
g(x) = \left \{
\begin{array}{ll}
x & x \in [0,1] \\
1 & x \in [1, +\infty].
\end{array}
\right.
\]
Then $g$ is a weak unit that is not strong, which shows that $G \notin \CR$. That $G \in \Y$ can be checked using \ref{2.1}(2) and some arithmetic.
\end{example}

\begin{example}\label{5.3}
$\CR \nRightarrow \Y$.

We give two (types of) examples.

(1) $C(K)$ for any $K$ infinite and almost-$P$ is $\CR$ not $\Y$, by \ref{5.1}.

(2) (Reconsidering \ref{5.2}(1)) Let $G \leq C(\aN)$ be generated by $F$, $a$, and $b$, where
\[
a(x) = \left \{
\begin{array}{ll}
\frac{1}{x} & x \mbox{ is even} \\
0 & x \mbox{ is odd}
\end{array}
\right.
\]
and
\[
b(x) = \left \{
\begin{array}{ll}
\frac{1}{x^2} & x \mbox{ is even} \\
0 & x \mbox{ is odd}.
\end{array}
\right.
\]
Observe that $g \in G$ can be put 
\[
(\ast) \, g \dot{=} r+sa + tb, \mbox{ for some } r,s,t \in \R,
\]
so that $\coz(g)$ differs by a finite set in $\N$ from either a clopen $U \ni \alpha$, or from the evens. Thus $G \in \CR$. $G \notin \Y$ since $\coz(a)=\coz(b)$ but $a \notin G(b)$.
\end{example}

\begin{example}
Some $G \in \M$.

Of course, $\HA \subseteq \M$, and any $F(K,\R) \in \HA$, and many ``weird" members of $\HA$ are exhibited in \cite{HK04}.

Beyond $\HA$, and resembling \ref{5.2}(1) and \ref{5.3}(2): let $G \leq C(\aN)$ be generated by $F$ and 
\[
a(x) = \left \{
\begin{array}{ll}
\frac{1}{x} & x \mbox{ is even} \\
0 & x \mbox{ is odd}.
\end{array}
\right.
\]
As before, $G \in \CR$. Also, $G \in \Y$ is visible from \ref{2.1}(5). So $G \in \M$.

(This kind of example can be manufactured replacing $\aN$ by any $K$ compact zero-dimensional with a proper dense cozero-set $S$. $S$ will be a disjoint union $\bigcup_n U_n$, with the $U_n$'s clopen. Then, ``split" $\{ U_n \}$ into ``evens and odds", and define $a$ accordingly.
\end{example}

\begin{question}
In $\Wstar$, what does $G \in \M$, $\Y$, or $\CR$ entail about the Yosida space $YG$?
\end{question}

%%%%%%%%%%%%%%%%%%%%%%%%%%%%%%%%%

%\newpage

\end{document}